\documentclass{amsart}
\usepackage{graphicx}
\usepackage{amsrefs}
\usepackage[colorlinks]{hyperref}
\newtheorem{theorem}{Theorem}[section]

\theoremstyle{definition}
\newtheorem{definition}[theorem]{Definition}
\newtheorem{remark}[theorem]{Remark}
\newtheorem{example}[theorem]{Example}

\numberwithin{equation}{section}

\begin{document}
\title[]{Coloring invariants for links in $\Sigma_g\times S^1$}%
\author{Zhiyun Cheng}
\author{Hongzhu Gao}
\address{School of Mathematical Sciences, Beijing Normal University, Beijing 100875, China}
\address{School of Mathematical Sciences, Beijing Normal University, Beijing 100875, China}
\email{czy@bnu.edu.cn}
\email{hzgao@bnu.edu.cn}

\begin{abstract}
Let $\Sigma_g$ be a closed oriented surface of genus $g$, in this paper we discuss how to define coloring invariants and its generalizations for links in $\Sigma_g\times S^1$.
\end{abstract}
\maketitle

\section{Introduction}\label{section1}
The notion of quandle was independently introduced by Joyce \cite{Joy1982} and Matveev \cite{Mat1982} in 1982. As an analogue of knot group, one can define the knot quandle (also called the fundamental quandle) $Q_K$ of a knot $K$. It is well known that two knots $K_1, K_2$ have isomorphic knot quandles if and only if $K_1=K_2$ or $K_1=rm(K_2)$, here $rm(K_2)$ denotes the mirror image of $K_2$ with its orientation reversed. For a given finite quandle $Q$, a homomorphism from $Q_K$ to $Q$ has an interpretation as colorings of knot diagrams of $K$ by elements of $Q$. The number of homomorphisms gives rise to a knot invariant, called the coloring invariant of $K$ with respect to $Q$ and usually denoted by Col$_Q(K)$. Quandle idea has attracted considerable interest and has been actively developed during the past forty years. For example, the homology theory of a rack and the cohomology theory of a quandle were introduced in \cite{FR1992} and \cite{CJKLS2003}, respectively. An enhanced version of Col$_Q(K)$, say the quandle cocycle invariant, can be derived from a quandle 2-cocycle.

As a generalization of quandle, which is a set equipped with a binary operation, a biquandle \cite{FJK2004} is a set equipped with two binary operations. Similar to the quandle coloring invariant, for a given finite biquandle $X$, one can assign elements of $X$ to semi-arcs of knot diagrams to define the biquandle coloring invariant Col$_{X}(K)$. It is worth pointing out that for classical knots, the knot quandle and the knot biquandle contain the same information \cite{Ish2020}. 

Let $\Sigma_g$ denote a closed oriented surface of genus $g$, the main aim of this paper is to define coloring invariants for links in $\Sigma_g\times S^1$. Each link in $\Sigma_g\times S^1$ admits an arrowed diagram on $\Sigma_g$. Two arrowed diagrams represent the same link if and only if they are related by a sequence of generalized Reidemeister moves \cite{DM2009}. For a given finite biquandle $X$ and a special automorphism $f\in\text{Aut}(X)$, we introduce a coloring which assigns elements of $X$ to semi-arcs of arrowed diagrams. It turns out that the number of such colorings provides a nonnegative integer-valued invariant for links in $\Sigma_g\times S^1$.

The article is organized as follows. In Section \ref{section2}, we give a quick review of the notion of biquandle and arrowed diagrams of links in $\Sigma_g\times S^1$. In Section \ref{section3}, we explain how to define the coloring of arrowed diagrams if we are given a finite biquandle and a special automorphism $f$. The proof of the invariance of the coloring number is also given in this section. The last section is devoted to discuss some generalizations of these coloring invariants.

\section{Biquandle and knots in $\Sigma_g\times S^1$}\label{section2}
\begin{definition}\label{definition}
A \emph{biquandle} is a nonempty set $X$ with two binary operations $\circ: X\times X\to X$ and $\ast: X\times X\to X$ satisfying the following three axioms:
\begin{enumerate}
\item $\forall x\in X, x\ast x=x\circ x$;
\item $\forall x, y\in X$, there are unique $z, w\in X$ such that $z\ast x=y$ and $w\circ x=y$, and the map $S: (x, y)\rightarrow(y\circ x, x\ast y)$ is invertible;
\item $\forall x, y, z\in X$, we have
\begin{center}
$(z\circ y)\circ(x\ast y)=(z\circ x)\circ(y\circ x)$,\\$(y\circ x)\ast(z\circ x)=(y\ast z)\circ(x\ast z)$,\\$(x\ast y)\ast(z\circ y)=(x\ast z)\ast(y\ast z)$.
\end{center}
\end{enumerate}
\end{definition}

A biquandle $X$ reduces to a \emph{quandle} if $x\circ y=x$ for any $x, y\in X$. Here we list some examples of biquandle.

\begin{example}
Every set $X$ equipped with $x\ast y=x\circ y=x$ is a biquandle. More generally, for a fixed bijection $\sigma: X\to X$, $X$ under $x\ast y=x\circ y=\sigma(x)$ is a biquandle.
\end{example}

\begin{example}
Let $R=\mathbb{Z}[t^{\pm1},s^{\pm1}]$, then $R$ is a biquandle if we define $x\ast y=tx+(s-t)y$ and $x\circ y=sx$. In general, any module $X$ over $R$ is a biquandle.
\end{example}

\begin{example}\label{example2.4}
Let $X=\{1, 2, 3\}$ and the two operations $\ast: X\times X\to X$ and $\circ: X\times X\to X$ are expressed as a $3\times 6$ matrix as follows.
\begin{center}
$\left(
\begin{array}{cccccc}
    2 & 2 & 1 & 2 & 2 & 1 \\
    1 & 1 & 2 & 1 & 1 & 2 \\
    3 & 3 & 3 & 3 & 3 & 3 \\
\end{array}
\right)$
\end{center}
Here the entry in row $x$ column $y$ denotes $x\ast y$ and the entry in row $x$ column $y+3$ denotes $x\circ y$ $(1\leq x, y\leq3)$.
\end{example}

Suppose we have two biquandles $(X, \ast_X, \circ_X)$ and $(Y, \ast_Y, \circ_Y)$, a map $f: X\to Y$ is called a \emph{homomorphism} if for any $x_1, x_2\in X$
\begin{center}
$f(x_1\ast_X x_2)=f(x_1)\ast_Yf(x_2)$ and $f(x_1\circ_X x_2)=f(x_1)\circ_Yf(x_2)$.
\end{center}
If a homomorphism is a bijection, then we say it is an \emph{isomorphism}. An isomorphism from $X$ to itself is called an \emph{automorphism} of $X$. We denote the set of all automorphisms of $X$ by Aut$(X)$.

Suppose we are given a finite biquandle $X$ and an oriented knot $K$ represented by an oriented knot diagram $D$. By a semi-arc, we mean a part of $D$ between two adjacent crossing points. We can assign each semi-arc of $D$ with an element of $X$, such that locally at each crossing point the relations in Figure \ref{figure1} are satisfied. Each assignment of $D$ with $X$ is called a \emph{coloring}. The three axioms in Definition \ref{definition} guarantee that the number of colorings Col$_X(D)$ is invariant under the three Reidemeister moves, therefore we obtain an integer-valued knot invariant Col$_X(K)$.

\begin{figure}[h]
\centering
\includegraphics{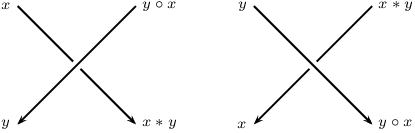}\\
\caption{The coloring rule of a biquandle}\label{figure1}
\end{figure}

Now we turn to the knot theory with ambient space $\Sigma_g\times S^1$. Let us consider $\Sigma_g\times S^1$ as the 3-manifold obtained from $\Sigma_g\times[0, 1]$ by identifying $\Sigma_g\times\{0\}$ and $\Sigma_g\times\{1\}$ with an identity. Let $L$ be an oriented link in $\Sigma_g\times S^1$, by a slight deformation if necessary, we can assume that $L$ is transverse to $\Sigma_g\times\{0\}$ and its projection $D$ in $\Sigma_g\times\{0\}$ has a finite number of double points. Each intersection point between $L$ and $\Sigma_g\times\{0\}$ corresponds to a dot in $D$. We assign an arrow to each dot (do not confuse it with the orientation of each component), which is directed from $\Sigma_g\times\{-\varepsilon\}(=\Sigma_g\times\{1-\varepsilon\})$ to $\Sigma_g\times\{+\varepsilon\}$. See Figure \ref{figure2} for an example. Finally, we obtain a link diagram with a finite number of arrows in $\Sigma_g\times\{0\}$, which is called an \emph{arrowed diagram} of $L$.

\begin{figure}[h]
\centering
\includegraphics{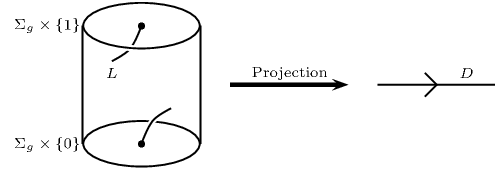}\\
\caption{Assigning an arrow to a dot}\label{figure2}
\end{figure}

A complete description of isotopy classes of links in $\Sigma_g\times S^1$ via arrowed diagrams was given by Dabkowski and Mroczkowski.

\begin{theorem}[\cite{DM2009}]
Two arrowed diagrams of links in $\Sigma_g\times S^1$ represent the same link if and only if they are related by a finite sequence of Reidemeister moves and the moves $\Omega_4$ and $\Omega_5$.
\end{theorem}

\begin{figure}[h]
\centering
\includegraphics{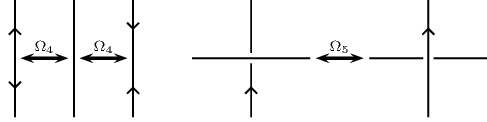}\\
\caption{Local move $\Omega_4$ and $\Omega_5$}\label{figure3}
\end{figure}

\section{Coloring invariants for links in $\Sigma_g\times S^1$}\label{section3}
In this section, we discuss how to define biquandle coloring invariants for knots in $\Sigma_g\times S^1$. The major challenge is how to deal with the arrows in arrowed diagrams. In our approach to this problem, we associate a special biquandle automorphism to each arrow. Let $X$ be a finite biquandle, $f$ be an automorphism of $X$, and $D$ an arrowed diagram representing a link $L$ in $\Sigma_g\times S^1$. By a \emph{semi-arc} of $D$, we mean a portion of $D$ between two adjacent crossing points, or an arrow and an adjacent crossing point, or two adjacent arrows. Now a \emph{coloring} of $D$ with respect to $(X, f)$ assigns each semi-arc of $D$ an element of $X$, such that the conditions depicted in Figure \ref{figure4} are satisfied. Let us use Col$_{(X, f)}(D)$ to denote the number of colorings of $D$. The following theorem tells us that for some particular $f\in$ Aut$(X)$, Col$_{(X, f)}(D)$ gives rise to a coloring invariant for links in $\Sigma_g\times S^1$.

\begin{figure}[h]
\centering
\includegraphics{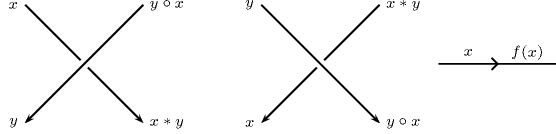}\\
\caption{Coloring rules}\label{figure4}
\end{figure}

\begin{theorem}\label{theorem3.1}
Let $X$ be a finite biquandle, $f\in$ Aut$(X)$ an automorphism of $X$ satisfying $x\ast y=x\circ f(y)$ for any $x, y\in X$, then Col$_{(X, f)}(D)$ does not depend on the choice of $D$.
\end{theorem}

\begin{proof}
The three axioms in the definition of biquandle guarantee that Col$_{(X, f)}(D)$ is invariant under the three Reidemeister moves. It is sufficient to consider the local moves $\Omega_4$ and $\Omega_5$.

For the local move $\Omega_4$, since $f$ is a bijection from $X$ to it self, one observes that there is a one-to-one correspondence between the colorings of the diagram before and after $\Omega_4$. See Figure \ref{figure5}. It follows that Col$_{(X, f)}(D)$ is invariant under $\Omega_4$.

\begin{figure}[h]
\centering
\includegraphics{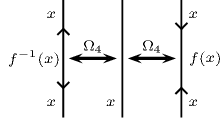}\\
\caption{Invariance under $\Omega_4$}\label{figure5}
\end{figure}

For the local move $\Omega_5$, notice that 
$$y\ast f^{-1}(x)=y\circ f(f^{-1}(x))=y\circ x,$$ 
and 
$$f(f^{-1}(x)\circ y)=x\circ f(y)=x\ast y.$$
It is not difficult to observe that $\Omega_5$ also preserves Col$_{(X, f)}(D)$.

\begin{figure}[h]
\centering
\includegraphics{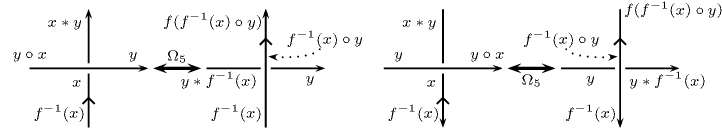}\\
\caption{Invariance under $\Omega_5$}\label{figure6}
\end{figure}
\end{proof}

Since the nonnegative integer Col$_{(X, f)}(D)$ does not depend on the choice of the diagram $D$, it defines a link invariant for links in $\Sigma_g\times S^1$. Let us use Col$_{(X, f)}(L)$ to denote it. Here we have some examples.

\begin{example}\label{example3.2}
Let us consider the two 2-component links $L_1, L_2$ in $T^2\times S^1$ depicted in Figure \ref{figure7}. Here we use the notation of virtual crossing point to denote the addition of a 1-handle to the 2-sphere where the link diagram lies in. In other words, these two link diagrams actually lie in $T^2$ and the virtual crossing points are not real. Let $X=\{1, 2, 3\}$ be the biquandle given in Example \ref{example2.4}, we choose an automorphism $f\in\text{Aut}(X)$ defined by $f(1)=2, f(2)=1, f(3)=3$. Direct calculation shows that Col$_{(X, f)}(L_1)=$ Col$_{(X, f)}(L_2)=5$. Thus both of them are nontrivial 2-component links in $T^2\times S^1$.

\begin{figure}[h]
\centering
\includegraphics{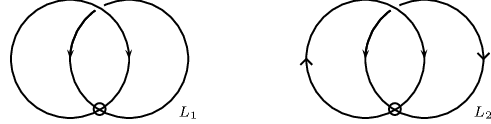}\\
\caption{Two 2-component links in $T^2\times S^1$}\label{figure7}
\end{figure}
\end{example}

\begin{example}\label{example3.3}
If $X$ is actually a quandle, in other words, for any $x, y\in X$ we have $x\circ y=x$, then it follows that $x\ast y=x\circ f(y)=x$. Hence $X$ is a trivial quandle of order $n$, if $|X|=n$. For arbitrary chosen $f\in S_n$, the symmetric group of degree $n$, then $f$ defines an automorphism of $X$ and the extra condition $x\ast y=x\circ f(y)$ in Theorem \ref{theorem3.1} is automatically satisfied. We decompose $X$ into orbits for the operation $f$, say $X=\bigcup\limits_{1\leq i\leq k}X_i$. Let $p: \Sigma_g\times S^1\to S^1$ be the projection from $\Sigma_g\times S^1$ to $S^1$ and $[\alpha]\in H_1(S^1)$ be the generator of $H_1(S^1)$. For a knot $K$ in $\Sigma_g\times S^1$, if $p_{\ast}([K])=h[\alpha]$ and a semi-arc is colored by $x$, then we have $f^h(x)=x$. It follows that $x$ can exactly be chosen from those $X_i$ which satisfies $|X_i|$ divides $h$.  As a consequence, we have Col$_{(X, f)}(K)=\sum\limits_{|X_i||h}|X_i|$. In particular, Col$_{(X, id)}(K)=n$ for any knot $K$ in $\Sigma_g\times S^1$.
\end{example}

\begin{remark}
In \cite{KM2005}, Kauffman and Manturov introduced the notions of virtual quandle and virtual biquandle, which are designed to define enhanced coloring invariants for virtual knots. In their paper, a \emph{virtual biquandle} is defined to be a biquandle $X$ endowed with an automorphism $f$. Hence the pair $(X, f)$ appeared in Theorem \ref{theorem3.1} can be regarded as a special case of virtual biquandle.
\end{remark}

\begin{remark}
A closely related algebraic structure, called \emph{labeled quandle}, was introduced by Kim in \cite{Kim2021}. A labeled quandle is a set equipped with two families of binary operations $\ast_i, \circ_i$ $(i\in\mathbb{Z})$. The index $i$ comes from a labelling of the link diagram by considering the projection $p: \Sigma_g\times S^1\to S^1$. Coloring invariants derived from labeled quandle were also discussed in \cite{Kim2021}. However, no concrete example of labeled quandle was given.
\end{remark}

\begin{remark}
It is well known that for knots in $S^3$, the quandle coloring invariant Col$_{X}(K)$ counts the homomorphisms from the knot quandle $Q_K$ to $X$. For knots in $\Sigma_g\times S^1$, Gabai proved that knots are determined by their complements when $g=0$ or 1 \cite{Gab1987}. Very recently, the general case for $g\geq2$ was confirmed by Cremaschi and Yarmola in \cite{CY2024}. It is an interesting question to find a similar interpretation for the coloring invariant Col$_{(X, f)}(K)$ introduced here.
\end{remark}

\section{Two extensions of coloring invariants}\label{section4}
\subsection{Biquandle 2-cocycle invariants}
An important extension of quandle coloring invariants comes from the (co)homology theory of quandles. This kind of (co)homology theory was first introduced for racks in \cite{FR1992} and later a modified version was introduced for quandles in \cite{CJKLS2003}. Roughly speaking, for a given quandle 2-cocycle and a fixed quandle coloring, one can assign a Boltzmann weight to each crossing point. Taking the product of Boltzmann weights for all crossing points, then the sum of all these products among all the colorings gives rise to a state-sum invariant for knots, called the \emph{quandle cocycle invariant}. The quandle coloring invariants can be recovered from the quandle cocycle invariants if the quandle 2-cocycle is a coboundary. For biquandles, Carter, Elhamdadi and Saito introduced a homology theory for the set-theoretic Yang-Baxter equations in \cite{CES2004}. The state-sum invariants can be similarly defined by using biquandle cocycles.

Now we take a quick review of the definition of biquandle cohomology groups. For a given biquandle $X$, let $C_n^{BR}(X)$ be the free abelian group generated by $n$-tuples $(x_1, \cdots, x_n)$, where $x_i\in X$. We define $d_n(x_1, \cdots, x_n)=$
\begin{center}
$\sum\limits_{i=1}^n(-1)^i((x_1, \cdots, x_{i-1}, x_{i+1}, \cdots, x_n)-(x_1\ast x_i, \cdots, x_{i-1}\ast x_i, x_{i+1}\circ x_i, \cdots, x_n\circ x_i))$ 
\end{center}
and extend it linearly to obtain a map $d_n: C_n^{BR}(X)\to C_{n-1}^{BR}(X)$. In particular, when $n\leq1$ we set $d_n=0$. It is a routine exercise to verify that $d_{n-1}\circ d_n=0$, thus $C_{\ast}^{BR}(X)=\{C_n^{BR}(X), d_n\}$ is a chain complex and the homology groups are called the \emph{birack homology groups}.

In order to define the biquandle homology groups, we need to consider the chain complex $C_{\ast}^{BQ}(X)=\{C_n^{BQ}(X), d_n\}$, where $C_n^{BQ}(X)=C_n^{BR}(X)/C_n^{BD}(X)$. Here $C_n^{BD}(X)$ denotes the free abelian group generated by $n$-tuples $(x_1, \cdots, x_n)$ where $x_i=x_{i+1}$ for some $1\leq i\leq n-1$. It turns out that $C_{\ast}^{BD}(X)=\{C_n^{BD}(X), d_n\}$ is a subcomplex of $C_{\ast}^{BR}(X)$ and the \emph{$n$-th biquandle homology group} and \emph{$n$-th biquandle cohomology group} are defined as
\begin{center}
$H_n^{BQ}(X)=H_n(C_{\ast}^{BQ}(X))$ and $H_{BQ}^n(X)=H^n(C_{BQ}^{\ast}(X))$,
\end{center}
where $C_{BQ}^{\ast}(X)=\text{Hom}(C_{\ast}^{BQ}(X), \mathbb{Z})$. For an abelian group $A$, the homology $H_n^{BQ}(X; A)$ and cohomology $H_{BQ}^n(X; A)$ with coefficient group $A$ can be similarly defined.

Let $X$ be a finite biquandle, $A$ an abelian group, $\phi$ a biquandle 2-cocycle and $f\in$ Aut$(X)$ an automorphism of $X$ satisfying $x\ast y=x\circ f(y)$ for any $x, y\in X$. For a link $L$ in $\Sigma_g\times S^1$, choose a link diagram $D$ of it. For a fixed $X$-coloring $\rho$ of $D$, consider a crossing points $c$ as in Figure \ref{figure1} and assign a Boltzmann weight $W_{\phi}^{\rho}(c)=\phi(x, y)^{w(c)}\in A$ to it. Here $w(c)$ denotes the sign of $c$.  We associate an element $\Phi_{\phi}^f(D)$ of the group ring $\mathbb{Z}[A]$ to $D$ as follows
\begin{center}
$\Phi_{\phi}^f(D)=\sum\limits_{\rho}\prod\limits_cW_{\phi}^{\rho}(c)\in\mathbb{Z}[A]$.
\end{center}

The following theorem suggests that for some suitable choice of $\phi$, $\Phi_{\phi}^f(D)$ does not depend on the choice of $D$. Hence it defines a link invariant for links in $\Sigma_g\times S^1$, we call it the \emph{biquandle cocycle invariant} of $L$ with respect to $f$ and $\phi$ and denote it by $\Phi_{\phi}^f(L)$.

\begin{theorem}\label{theorem4.1}
Let $X$ be a finite biquandle, $A$ be an abelian group, $f\in$ Aut$(X)$ be an automorphism of $X$ satisfying $x\ast y=x\circ f(y)$ for any $x, y\in X$, and $\phi$ be a biquandle 2-cocycle satisfying $\phi(x, y)\phi(y, f^{-1}(x))=1$ for any $x, y\in X$, then $\Phi_{\phi}^f(D)$ does not depend on the choice of $D$.
\end{theorem}
\begin{proof}
By definition, $\phi$ is a biquandle 2-cocycle if it satisfies
\begin{enumerate}
  \item $\phi(x, x)=1$ for any $x\in X$;
  \item $\phi(x, y)\phi(y, z)\phi(x\ast y, z\circ y)=\phi(x\ast z, y\ast z)\phi(y\circ x, z\circ x)\phi(x, z)$ for any $x, y, z\in X$.
\end{enumerate}
In order to prove the theorem, it suffices to show that for a fixed coloring $\rho$, $\prod\limits_cW_{\phi}^{\rho}(c)$ is invariant under the Reidemeister moves and local moves $\Omega_4$ and $\Omega_5$. 

For the first Reidemeister move and the third Reidemeister move, it is well known that the two properties (1) and (2) above guarantee the invariance of $\prod\limits_cW_{\phi}^{\rho}(c)$. Since the two crossing points involved in the second Reidemeister move have opposite signs, the contributions from these two crossing points to $\prod\limits_cW_{\phi}^{\rho}(c)$ cancel out.

For the local move $\Omega_4$, it has no effect on the Boltzmann weights of crossing points of $D$. Thus it preserves $\prod\limits_cW_{\phi}^{\rho}(c)$.

For the local move $\Omega_5$, let us consider the two cases depicted in Figure \ref{figure6}. For the first case, the Boltzmann weight of the left crossing point equals $\phi(x, y)$ and that of the right crossing point equals $\phi(y, f^{-1}(x))^{-1}$, they are equal according to the assumption $\phi(x, y)\phi(y, f^{-1}(x))=1$. For the second case, one observes that the Boltzmann weights of the two crossing points are $\phi(x, y)^{-1}$ and $\phi(y, f^{-1}(x))$ respectively, which also coincide.
\end{proof}

Obviously, if one sets $\phi(x, y)=1$ for any $x, y\in X$, then $\Phi_{\phi}^f(K)$ reduces to the coloring invariant Col$_{(X, f)}(K)$.

\begin{example}\label{example4.2}
Consider the two links $L_1, L_2$, the biquandle $X$ and the automorphism $f\in\text{Aut}(X)$ given in Example \ref{example3.2}, let us choose an infinite cyclic group $A$ generated by $a$ and a biquandle 2-cocycle $\phi$ defined by
\begin{center}
$\phi(1, 1)=\phi(2, 2)=\phi(3, 3)=\phi(1, 2)=\phi(2, 1)=1$,
\end{center}
and
\begin{center}
$\phi(1, 3)=\phi(2, 3)=a, \phi(3, 2)=\phi(3, 1)=a^{-1}$.
\end{center}
It is easy to see that $\phi$ satisfies the extra condition required in Theorem \ref{theorem4.1}. Direct calculation shows that
\begin{center}
$\Phi_{\phi}^f(L_1)=2a+2a^{-1}+1$ and $\Phi_{\phi}^f(L_2)=5$.
\end{center}
Comparing with the result given in Example \ref{example3.2}, we find that the biquandle cocycle invariants are stronger than the biquandle coloring invariants.
\end{example}

For classical knots or virtual knots, it is well known that if the 2-cocycle $\phi$ is a coboundary, then the cocycle invariant reduces to the coloring invariant. In our case, if $\phi=\delta_1\psi$ for some $\psi\in C^1_{BQ}(X)$, then 
\begin{center}
$\phi(x, y)=\delta_1\psi(x, y)=\psi d_2(x, y)=\psi(y)^{-1}\psi(y\circ x)\psi(x)\psi(x\ast y)^{-1}$.
\end{center}
For a fixed coloring $\rho$, the Boltzmann weight of each crossing point splits into four terms assigned to the four semi-arcs around it. If a semi-arc contains no arrow, then the contributions received by this semi-arc from its two ends cancel out. Hence if the arrowed diagram actually contains no arrow, the product $\prod\limits_cW_{\phi}^{\rho}(c)=1$. If the arrowed diagram contains some arrows, notice that 
\begin{center}
$1=\phi(x, y)\phi(y, f^{-1}(x))=\delta_1\psi(x, y)\delta_1\psi(y, f^{-1}(x))$,
\end{center}
which implies 
\begin{center}
$1=\psi(y)^{-1}\psi(y\circ x)\psi(x)\psi(x\ast y)^{-1}\psi(f^{-1}(x))^{-1}\psi(f^{-1}(x)\circ y)\psi(y)\psi(y\ast f^{-1}(x))^{-1}$.
\end{center}
Together with the assumption that $y\ast f^{-1}(x)=y\circ x$, it follows that 
\begin{center}
$\psi(x)\psi(x\ast y)^{-1}\psi(f^{-1}(x))^{-1}\psi(f^{-1}(x)\circ y)=1$.
\end{center}
Or equivalently, 
\begin{center}
$\psi(x)\psi(f^{-1}(x))^{-1}=\psi(x\ast y)\psi(f^{-1}(x)\circ y)^{-1}$.
\end{center}
Since the elements $x, y$ can be arbitrarily chosen and $f$ is a bijection, the equality above can be rewritten as
\begin{center}
$\psi(f(x))\psi(x)^{-1}=\psi(f(x)\ast f(y))\psi(x\circ f(y))^{-1}=\psi(f(x\ast y))\psi(x\ast y)^{-1}$,
\end{center}
or
\begin{center}
$\psi(f(x))\psi(x)^{-1}=\psi(f(x)\ast y)\psi(x\circ y)^{-1}=\psi(f(x\circ y))\psi(x\circ y)^{-1}$.
\end{center}
Similar to the notion of Latin quandle, we say a biquandle $X$ is \emph{Latin} if for any $x\in X$, the maps $x\ast\bullet : X\to X$ and $x\circ\bullet : X\to X$ are both bijective. If either one of these two maps is bijective, then we say $X$ is \emph{semi-Latin}. Now assume we are given a finite semi-Latin biquandle $X$, then the two equalities above suggest that for any $x, y\in X$, we have 
\begin{center}
$\psi(f(x))\psi(x)^{-1}=\psi(f(y))\psi(y)^{-1}$.
\end{center}
That is to say, for any $x\in X$, $\psi(f(x))\psi(x)^{-1}$ is a constant element of the abelian group $A$. Let us use $a$ to denote it. Similar to Example \ref{example3.3}, we use $p: \Sigma_g\times S^1\to S^1$ to denote the projection from $\Sigma_g\times S^1$ to $S^1$ and $[\alpha]\in H_1(S^1)$ be the generator of $H_1(S^1)$. For an oriented link $L$ in $\Sigma_g\times S^1$, if $p_{\ast}([L])=k[\alpha]$ and the 2-cocycle $\phi$ is a coboundary, then we have $\Phi_{\phi}^f(L)=\sum\limits_{\rho}a^k$. When $k=0$, this recovers the corresponding result of classical knots.

\subsection{Index type invariants}
In this subsection, we discuss another application of the quandle coloring invariants. Roughly speaking, what we want to do is to introduce an enhanced version of the writhe of a link diagram. Recall that for a given link diagram $D$, the writhe of $D$ is the sum of the signs among all the crossing point. This is not a link invariant in general, since it is not preserved under the first Reidemeister move. In order to overcome this, we can assign an index (e.g. an integer, a polynomial or an element of a group, etc) to each crossing point of $D$ and the writhe now turns into a weighted sum. The requirements of the index that make this weighted sum becomes a link invariant were first introduced in \cite{Che2021} for virtual knots. Here we present a modified version as follows.

\begin{definition}\label{definition4.4}
An \emph{index} is an assignment which assigns an index to each crossing point of a link diagram, such that
\begin{enumerate}
  \item the index of the crossing involved in the first Reidemeister move is fixed;
  \item the indices of the two crossings involved in the second Reidemeister move are the same;
  \item the indices of the three crossings involved in the third Reidemeister move are preserved respectively;
  \item the index of any crossing which does not appear in a Reidemeister move is preserved under this Reidemeister move.
\end{enumerate}
\end{definition}

Index type invariants of virtual knots have been intensively studied during the past ten years. The readers are referred to \cite{CFGMX2020} for a survey of some recent progress on this topic. 

Let $L$ be link in $\Sigma_g\times S^1$ and $D$ a link diagram of it. Assume we are given a finite biquandle $X$ and an automorphism $f\in$ Aut$(X)$ satisfying $x\ast y=x\circ f(y)$ for all $x, y\in X$, a concrete example of index can be constructed as follows. Consider the group
\begin{center}
$G_X=<(x, y)\in X\times X|(x, x)=(x, y)(y, f^{-1}(x))=1, (x, y)=(x\ast z, y\ast z), (y, z)=(y\circ x, z\circ x), (x, z)=(x\ast y, z\circ y)>$.
\end{center}
Note that in $G_X$ we also have
\begin{center}
$(x\circ y, z\ast y)=((x\circ y)\ast^{-1}y, z)=((x\circ y)\ast^{-1}y, (z\circ y)\circ^{-1}y)=(x\circ y, z\circ y)=(x, z)$.
\end{center}
On the other hand, for a given abelian group $A$, any homomorphism $\phi$ from $G_X$ to $A$ satisfies the biquandle 2-cocycle condition mentioned in Theorem \ref{theorem4.1}. Furthermore, for any such biquandle 2-cocycle $\phi$, not only the product of the Boltzmann weights of the three crossings involved in the third Reidemeister move is preserved, actually, now the Boltzmann weight of each crossing is preserved under the third Reidemeister move. 

For a fixed coloring of $D$ by $X$, say $\rho$, we associate a weight $w_{\rho}(c)=(x, y)\in G_X$ to each of the two crossing points in Figure \ref{figure1}. Now we define the \emph{index} (associated with $X$ and $f$) of a crossing point $c$ to be Ind$(c)=\sum_{\rho}w_{\rho}(c)\in\mathbb{Z}G_X$. In particular, if there exists a homomorphism $\phi$ from $G_X$ to an abelian group $A$, then we obtain an induced index $\phi(\sum_{\rho}w_{\rho})\in\mathbb{Z}A$. Now for any element $g\in\mathbb{Z}G_X$, we define
\begin{center}
$a_g(D)=
\begin{cases}
\sum\limits_{\text{Ind}(c)=g}w(c),& \text{if }g\neq\sum\limits_{\rho}1;\\
\sum\limits_{\text{Ind}(c)=g}w(c)-w(D),& \text{if }g=\sum\limits_{\rho}1.
\end{cases}$
\end{center}
Here $w(D)$ denotes the writhe of $D$.

\begin{theorem}
For any $g\in\mathbb{Z}G_X$, the integer $a_g(D)$ does not depend on the choice of $D$, hence it defines a link invariant $a_g(L)$ for links in $\Sigma_g\times S^1$.
\end{theorem}

\begin{proof}
The result mainly follows from the fact that Ind$(c)$ satisfies all the conditions mentioned in Definition \ref{definition4.4}, which can be verified straightforwardly. As a consequence, $a_g(D)$ is invariant under the three classical Reidemeister moves. Obviously, the local move $\Omega_4$ has no effect on the index of each crossing point, hence also the integer $a_g(D)$. The invariance of the index of the crossing point involved in $\Omega_5$ follows directly from the fact $(x, y)(y, f^{-1}(x))=1$ assumed in the definition of $G_X$. Therefore, the integer $a_g(D)$ is also preserved under $\Omega_5$.
\end{proof}

\begin{remark}
Although the group ring $\mathbb{Z}G_X$ is infinite, for each link $L\subset\Sigma_g\times S^1$, only finitely many $a_g(L)$ are nonzero.
\end{remark}

\begin{example}\label{example4.6}
Consider the two links $L_1, L_2$ discussed in Example \ref{example3.2}, the biquandle $X$ and the automorphism $f$ used in Example \ref{example4.2}. It is not difficult to find that $G_X=\mathbb{Z}$, which is generated by $a=(1, 3)=(2, 3)$. Then we have
\begin{center}
$a_g(L_1)=
\begin{cases}
-1,& \text{if }g=2a+2a^{-1}+1;\\
0,& \text{otherwise},
\end{cases}$
\end{center}
and
$a_g(L_2)=0$ for any $g\in\mathbb{Z}G_X$.
\end{example}

\begin{example}\label{example4.7}
Consider the 2-component link $L_3\subset S^2\times S^1$ depicted in Figure \ref{figure8}. We still use the same biquandle $X$ and automorphism $f$ as that in Example \ref{example4.2}. Now direct calculation shows that 
\begin{center}
$a_g(L_3)=
\begin{cases}
-1,& \text{if }g=1+2a;\\
-1,& \text{if }g=1+2a^{-1};\\
0,& \text{otherwise}.
\end{cases}$
\end{center}
We would like to remark that if we choose the biquandle 2-cocycle $\phi$ in Example \ref{example4.2}, the biquandle cocycle invariant $\Phi_{\phi}^f(L_3)$ equals 3, which contains no extra information comparing with the biquandle coloring invariant.
\begin{figure}[h]
\centering
\includegraphics{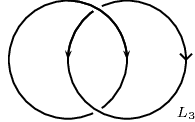}\\
\caption{A 2-component link $L_3\subset S^2\times S^1$}\label{figure8}
\end{figure}
\end{example}

\section*{Acknowledgement}
Part of this paper was written when the authors were visiting Guangxi Normal University in 2023. We wish to thank Mengjian Xu for useful conversations. The first author is grateful to Faze Zhang for informing him the paper \cite{Kim2021}. The authors were supported by the NSFC grant 12371065, NSFC grant 12071034 and NSFC grant 12271040.

\bibliographystyle{amsplain}
\bibliography{}

\end{document}